\documentclass[a4paper,12pt]{amsart}
\usepackage{amssymb}
\usepackage{ifthen}
\usepackage{graphicx}
\usepackage{float}
\usepackage{caption}
\usepackage{subcaption}
\usepackage{cite}
\usepackage{amsfonts}
\usepackage{amscd}
\usepackage{amsxtra}
\usepackage{color}

\usepackage[latin1]{inputenc}

\newtheorem{thm}{Theorem}
\newtheorem{cor}{Corollary}
\newtheorem{lem}{Lemma}
\newtheorem{rem}{Remark}
\newtheorem{prop}{Proposition}

\newtheorem*{conj}{Conjecture}
\newtheorem{prob}{Problem}

\theoremstyle{definition}
\newtheorem{defn}{Definition}[section]
\newtheorem{example}{Example}

\newenvironment{pf}[1][]{%
	\vskip 1mm
	\noindent
	\ifthenelse{\equal{#1}{}}%
	{{\slshape Proof. }}%
	{{\slshape #1.} }%
}%
{\qed\bigskip}

\newcounter{alphabet}



\def\be{\begin{equation}}
	\def\ee{\end{equation}}

\newcommand{\bee}{\begin{enumerate}}
	\newcommand{\eee}{\end{enumerate}}

\newcommand{\blem}{\begin{lem}}
	\newcommand{\elem}{\end{lem}}
\newcommand{\bthm}{\begin{thm}}
	\newcommand{\ethm}{\end{thm}}
\newcommand{\bcor}{\begin{cor}}
	\newcommand{\ecor}{\end{cor}}
\newcommand{\beg}{\begin{example}}
	\newcommand{\eeg}{\end{example}}
\newcommand{\begs}{\begin{examples}}
	\newcommand{\eegs}{\end{examples}}
\newcommand{\bdefe}{\begin{defn}}
	\newcommand{\edefe}{\end{defn}}
\newcommand{\bprob}{\begin{prob}}
	\newcommand{\eprob}{\end{prob}}
\newcommand{\bques}{\begin{ques}}
	\newcommand{\eques}{\end{ques}}
\newcommand{\bei}{\begin{itemize}}
	\newcommand{\eei}{\end{itemize}}
\newcommand{\bcon}{\begin{conj}}
	\newcommand{\econ}{\end{conj}}
\newcommand{\bcons}{\begin{conjs}}
	\newcommand{\econs}{\end{conjs}}
\newcommand{\bprop}{\begin{propo}}
	\newcommand{\eprop}{\end{propo}}
\newcommand{\br}{\begin{rem}}
	\newcommand{\er}{\end{rem}}
\newcommand{\brs}{\begin{rems}}
	\newcommand{\ers}{\end{rems}}
\newcommand{\bo}{\begin{obser}}
	\newcommand{\eo}{\end{obser}}
\newcommand{\bos}{\begin{obsers}}
	\newcommand{\eos}{\end{obsers}}
\newcommand{\bpf}{\begin{pf}}
	\newcommand{\epf}{\end{pf}}
\newcommand{\ba}{\begin{array}}
	\newcommand{\ea}{\end{array}}
\newcommand{\beq}{\begin{eqnarray}}
	\newcommand{\beqq}{\begin{eqnarray*}}
		\newcommand{\eeq}{\end{eqnarray}}
	\newcommand{\eeqq}{\end{eqnarray*}}

\newcounter{minutes}
\divide\time by 60
\newcounter{hours}
\multiply\time by 60 \addtocounter{minutes}{-\time}

\begin{document}
	\title[Dual Smale's mean value conjecture, n=7]{Dual Smale's mean value conjecture for $n=7$}

	\author{Aimo Hinkkanen, Ilgiz R. Kayumov,  Diana M. Khammatova}
	
		\address{Aimo Hinkkanen,
				Department of Mathematics
				University of Illinois at Urbana-Champaign
				1409 W. Green Street
				Urbana, Illinois 61801-2975
			}
		\email{aimo@math.uiuc.edu}
	
	\address{I.R. Kayumov, Kazan Federal University, Kremlevskaya 18, 420 008 Kazan, Russia
	}
	\email{ikayumov@kpfu.ru}
	
	\address{D.M. Khammatova, Moscow Polytechnic University, Bolshaya Semyonovskaya 38, 107023 Moscow, Russia; Kazan Federal University, Kremlevskaya 18, 420 008 Kazan, Russia
	}
	\email{dianalynx@rambler.ru}

	

	\begin{abstract}
		We give an analytic proof of the dual Smale's mean value conjecture in the case $n=7$.
	\end{abstract}

	\maketitle
	\pagestyle{myheadings}
	 
	\markboth{A. Hinkkanen, I.R. Kayumov, D.M. Khammatova}{Dual Smale's mean value conjecture, n=7}
	
	\section{Introduction}
	In this paper we study a problem closely related to  Smale's mean value conjecture.
	We consider the  class of polynomials defined by
	$$
	\mathfrak{P}_n = \{z+\sum\limits_{k=2}^n c_k z^k \colon c_k \in \mathbb{C}, c_n\not= 0\} ,
	$$
where $n\geq 2$.	In other words, it is the class of all polynomials $f(z)$ of degree $n$ such that $f(0) = 0$ and $f'(0)=1$. We denote the critical points of $f$, that is, the zeros of $f'$, by $\zeta_k$ where $1\leq k\leq n-1$, and generically by $\zeta$.
	
	For each $f\in\mathfrak{P}_n$ we write
	$$
	T(f) = \min \left\{\left|\frac{f(\zeta)}{\zeta}\right|  \colon \,\,f'(\zeta)=0\right\}
	$$
	and
	$$
	S(f) = \max \left\{\left|\frac{f(\zeta)}{\zeta}\right|  \colon \,\,f'(\zeta)=0\right\}.
	$$
	In 1981, Smale \cite{s} proved the estimate $T(f)<4$ and conjectured that $\max\limits_{f\in \mathfrak{P}_n} T(f) = 1-\frac 1 n$. It is easy to check that for $f(z) = z+c z^n$, where $c\not = 0$ we obtain $T(f)=1-\frac 1 n$. Therefore, it is known that
	$$
	1-\frac{1}{n} \leq T(f) < 4.
	$$
	There exist a lot of estimates for various special cases. Some of them  can found in the papers \cite{bmn}, \cite{hk}.
	
	The main object of the current research is the conjecture known as the dual Smale's conjecture.
	
	\begin{conj}
		$\min\limits_{f\in \mathfrak{P}_n} S(f) = \frac 1 n$.
	\end{conj}
	
	This conjecture was formulated by Dubinin and Sugawa in \cite{d} and independently by Ng.
	
	Tischler \cite{t} proved this conjecture for all  polynomials such that 
	$$
	\frac{f(\zeta)}{\zeta}= {\rm Const},\qquad  f'(\zeta)=0.
	$$
	Polynomials satisfying this condition are called conservative polynomials.
	
	Dubinin and Sugawa \cite{d} proved the existence of a critical point $\zeta$ such that
	$$
	\left|\frac{f(\zeta)}{\zeta}\right|\ge\frac {1}{n4^n}.
	$$
	In 2018 Dubinin \cite{dub18} proved that for every polynomial $f\in\mathfrak{P}_n$, $n\geq 2$, and every point $z$ there exists a critical point $\zeta$ of $f$ such that 
	$$
\left|	\frac{f(z)-f(\zeta)}{z-\zeta}\right|\geq \frac1n\tan\frac{\pi}{4n}|f'(z)|.
	$$ 
	The author notes that $\tan\frac{\pi}{4n}$ in this inequality can be replaced by $\frac1{n}$.

	For small values of $n$ the conjecture can be checked directly, but even for $n=4$ the problem becomes quite complicated. 
	
	The proof for the case $n=4$ was given by Tischler in \cite{t}. It was demonstrated that there exists a  critical point $\zeta$ of the polynomial $f\in\mathfrak{P}_n$ such that
	$$
	\left|\frac{f(\zeta)}{\zeta}-\frac{1}{2}\right| \leq \frac{1}{2}-\frac{1}{n}, \quad 2 \leq n \leq 4.
	$$
	However, later Tyson \cite{tyson} showed that for $n\geq 5$ this inequality does not hold.
	
	In 2019 the authors \cite{hkk19} proved the conjecture for $n=5$ and $n=6$. Numerical experiments showed that the same approach can be used for $n=7$ as well, but the calculations seemed to be too complicated. 
	
	In the present work we prove the dual Smale's conjecture for $n=7$. Thus we obtain the following theorem.

\begin{thm} \label{th0}
Suppose that $f$ is a polynomial of degree $7$ with $f(0)=0$ and $f'(0)=1$. Then there is a point $\zeta$ such that $f'(\zeta)=0$ and $|f(\zeta)/\zeta|\geq 1/7$. 

If there is no point $\zeta$ such that $f'(\zeta)=0$ and $|f(\zeta)/\zeta|> 1/7$, then $f$ is of the form 
\begin{equation} \label{eq1}
f(z) = \frac{1}{7a} (  1 - (1-az)^7    ) 
\end{equation}
where $a\in {\mathbb C}\setminus \{0\}$. 
\end{thm}
	
	The paper is organized as follows. In the second section we briefly explain the main concept of the proof. In the third section we derive Theorem~\ref{th0} from the technical main result, Theorem~\ref{t1}. The fourth section contains some additional facts, which are essential for the proof. The main theorem, which gives the necessary result, is proved in the fifth section. In the sixth section we give a short summary of the current research. 
	
	\section{The main idea}
	
	Our approach is based on a fairly simple strategy: choose a critical point $\zeta$ whose modulus is minimal among all critical points of the given polynomial.
	
	 To be more precise, let us formulate an auxiliary problem. 
	\begin{prob}\label{pr1}
		Let $\alpha_f = \min\limits_{f'(\zeta)=0}|\zeta|$,
		\begin{equation} \label{L1}
		\Lambda_f = \max\left\{\left|\frac{f(\zeta)}{\zeta}\right|: \quad f'(\zeta)=0, |\zeta| = \alpha_f \right\}.
		\end{equation} 
		Find
		\begin{equation} \label{L2}
		\Lambda = \min\limits_{f\in\mathfrak{P}_7}\Lambda_f 
		\end{equation}
		and find the extremal functions.		
	\end{prob}
	It is obvious that if  $\Lambda=1/7$ then the dual Smale's mean value conjecture for $n=7$ is true.

	Let $f\in \mathfrak{P}_7$. Then $f$  has exactly six critical points $\zeta_1,\dots ,\zeta_6$, counting multiplicities. We note that one can replace $f(z)$ by the function $\frac{f(az)}a$ where $a\in {\mathbb C}\setminus \{0\}$. Therefore, without loss of generality we can assume that $\zeta_6=1$ and $|\zeta_k|\geq 1$ for $1\leq k\leq 5$. 
	
	We use the integral representation
	\begin{equation} \label{ee1}
	\frac {f(w)}w=\int\limits_0^1(1-tw)\prod\limits_{j=1}^5\left(1-\frac{tw}{\zeta_j}\right)\,dt.
	\end{equation} 
	Now we write $\frac1{\zeta_j}=z_j$, $|z_j|\leq 1$, and calculate the expression at the point $w=1$:
	$$
	\frac {f(1)}1=\int\limits_0^1(1-t)(1-z_1t)(1-z_2t)(1-z_3t)(1-z_4t)(1-z_5t)\,dt.
	$$
	We will prove the following fact.
	
	\begin{prop}\label{n0}
		Suppose that $f\in\mathfrak{P}_7$, and let the points $\zeta_k:=\frac1{z_k}$, $1\leq k\leq 5$ and $z_6=1$  be the critical points of $f$, where $|z_k|\leq 1$ for $1\leq k\leq 5$. Then $f(1)\not = 0$.
	\end{prop}
	
	Then we can use the minimum principle in the same way as in \cite{hkk19} and consider only points on the boundary, i.e., points $z_k$ such that $|z_k|=1$ for $1\leq k\leq 5$. 
	
	Therefore, to solve Problem 1 it is sufficient to prove the following theorem.
	
	\begin{thm}\label{t1}
		Let $z_1,z_2,z_3,z_4,z_5\in\mathbb C$ be complex numbers lying on the unit circle. Then
		\begin{equation}\label{source}
			S=\left|\int\limits_0^1(1-t)(1-z_1t)(1-z_2t)(1-z_3t)(1-z_4t)(1-z_5t)\,dt\right|\ge\frac 1 7.
		\end{equation}
	The equality holds if, and only if, 
	$$
	z_i=1,\,\,1\leq i\leq 5\qquad\text{or}\qquad z_i=e^{\frac{i\pi }3},\,\,1\leq i\leq 5 \qquad\text{or}\qquad z_i=e^{-\frac{i\pi }3},\,\,1\leq i\leq 5.
	$$
	\end{thm}

\section{Proof of Theorem~1}

We show how one can prove Theorem~\ref{th0} assuming the validity of Theorem~\ref{t1}. Now Theorem~\ref{t1} shows that $\Lambda$ defined by (\ref{L2}) satisfies $\Lambda\geq 1/7$.  To find all $f$ with $\Lambda_f=1/7$, for $\Lambda_f$ defined by (\ref{L1}), we may assume that $f$ is given by (\ref{ee1}). Recall that the critical points of $f$ are $1$ and $\zeta_j=1/z_j$. 

If $z_j=1$ for all $j$ with $1\leq j\leq 5$, then
\begin{equation} \label{ff}
f(z) = z \int_0^1 (1-tz)^6\, dt = \frac{1}{7}  \left(  1 - (1-z)^7   \right)   .
\end{equation}
At the only critical point $\zeta=1$ of $f$ we have $f(\zeta)/\zeta = 1/7$ so that $\Lambda_f=1/7$.  

If  $z_j=q$ for all $j$ with $1\leq j\leq 5$, where $q=e^{  i \pi /3}$  or $q=e^{ - i \pi /3}$, so that $1-q=\overline{q}$ and $(1-q)^6= 1$, then 
by (\ref{ee1}) and integration by parts, 
\begin{eqnarray*}  
f(z) & = & 
z \int_0^1 (1 - tz) (1-tqz)^5\, dt = \frac{-1}{6qz}  z \int_0^1 (1 - tz) \, d(1-tqz)^6
 \\   & = & 
 \frac{-1}{6q}  \left(  (1 - z)  (1-qz)^6 -1   +  \frac{-1}{7q} (  (1-qz)^7  -1  )    \right)
 \\   & = & 
  \frac{1}{42q^2}  \left(  (7q-1) (1 - (1-qz)^6)  + 6qz (1-qz)^6  \right) .
\end{eqnarray*}
Now the distinct critical points of $f$ are $1$ and $1/q$. When $q=e^{ \pm i \pi /3}$, we have
$$
\frac{f(1)}{1} = \frac{ (7q - 1 ) + (1-q)^7  }{42q^2}
= \frac{ 6q   }{42q^2} = \frac{  1  }{ 7 q}
$$ 
so that $|f(1)/1|=1/7$, while 
$$
\frac{f(1/q)}{1/q} = \frac{7q-1}{42q^2}  ,
$$
hence
$$
\left| \frac{f(1/q)}{1/q}  \right| = \frac{  \sqrt{ 50 - 14 {\rm Re}\, q    }  }{42}
= \frac{  \sqrt{ 43   }  }{42} \approx 0.156 > \frac{1}{7} .
$$
So when $z_j=q$ for all $j$ with $1\leq j\leq 5$, where $q=e^{  i \pi /3}$  or $q=e^{ - i \pi /3}$,  we have $\Lambda_f>1/7$. 

Hence we have $\Lambda_f = 1/7$  if, and only if, there is $a\in {\mathbb C} \setminus \{0\}$ such that $f(az)/a$ is given by the right hand side of (\ref{ff}).  This proves Theorem~\ref{th0}. 
	
	\section{Proof of the Proposition 1}
	We will need the following  lemma.
	
	\begin{lem}\label{lem1}
		Let 
		$$a  =\int\limits_0^1t(1-t)(1-z_1t)(1-z_2t)(1-z_3t)(1-z_4t)\,dt,$$
		$$b =\int\limits_0^1(1-t)(1-z_1t)(1-z_2t)(1-z_3t)(1-z_4t)\,dt,$$
		where $|z_k|=1$, $1\leq k\leq 4$.
		Then $|b|>|a|$.
	\end{lem}
\begin{proof}
		We represent each $z_k$ as $z_k=x_k+iy_k$, $x_k^2+y_k^2=1$. Then
		
		\begin{multline*}
			11025(|b|^2-|a|^2) =3730+1064 x_1 x_2 x_3 x_4 - 2359 (x_1+ x_2+x_3+x_4)+ 84 y_1 y_2 y_3 y_4 +\\+ 1701 (x_1 x_2 + x_1 x_3 + x_2 x_3+x_1 x_4 + x_2 x_4 + x_3 x_4) - 329( y_1 y_2 + y_1 y_3 + y_2 y_3+y_1 y_4 + y_2 y_4 + y_3 y_4) - \\ - 1316 (x_1 x_2 x_3 +x_1 x_2 x_4+x_1 x_3 x_4 + x_2 x_3 x_4) + 238 (x_3 y_1 y_2 + x_4 y_1 y_2 +x_2 y_1 y_3+\\ + x_4 y_1 y_3+x_1 y_2 y_3 + x_4 y_2 y_3  +x_2 y_1 y_4 + x_3 y_1 y_4  + x_1 y_2 y_4 + x_3 y_2 y_4+ x_1 y_3 y_4 + x_2 y_3 y_4) -\\- 182 (x_3 x_4 y_1 y_2  + x_2 x_4 y_1 y_3 + x_1 x_4 y_2 y_3  + x_2 x_3 y_1 y_4  + x_1 x_3 y_2 y_4 +  x_1 x_2 y_3 y_4) 
		\end{multline*}
		We denote the right-hand side of the resulting expression by $h$. It is evident that it will be not less than
		$$
		h_1=h - 7 ( (y_1 - y_2)^2 ( y_3 - y_4)^2 + (y_1 - y_3)^2 ( 
		y_2 - y_4)^2 + (y_1 - y_4)^2 ( y_3 - y_2)^2).
		$$
		Taking into account the equality $x_k^2+y_k^2=1$, we obtain a function that contains expressions in the form
		$$
		-(301 - 238 x_3 + 14 x_3^2 + (-238 + 182 x_3) x_4 + 14 x_4^2)y_1 y_2
		$$
		It can be easily seen that the expression in brackets is non-negative, so
		\begin{multline*}
			h_1\geq h_1 - \frac12 (301 - 238 x_3 + 14 x_3^2 - 238 x_4 + 182 x_3 x_4 + 14 x_4^2) (y_1 - y_2)^2 - \\-
			\frac12 (301 - 238 x_2 + 14 x_2^2 - 238 x_4 + 182 x_2 x_4 + 
			14 x_4^2) (y_1 - y_3)^2 - \\ -
			\frac12 (301 - 238 x_2 + 14 x_2^2 - 238 x_3 + 182 x_2 x_3 + 
			14 x_3^2) (y_1 - y_4)^2 - \\ -
			\frac12 (301 - 238 x_1 + 14 x_1^2 - 238 x_4 + 182 x_1 x_4 + 
			14 x_4^2) (y_2 - y_3)^2 - \\ -
			\frac12 (301 - 238 x_1 + 14 x_1^2 - 238 x_3 + 182 x_1 x_3 + 
			14 x_3^2) (y_2 - y_4)^2 - \\ -
			\frac12 (301 - 238 x_1 + 14 x_1^2 - 238 x_2 + 182 x_1 x_2 + 
			14 x_2^2) (y_3 - y_4)^2	
		\end{multline*}
Using the equality  $x_k^2+y_k^2=1$ one more time, we obtain an expression that depends only on $x_k$, $1\leq k\leq 4$.

We replace $x_k = 1-2w_k$, $0\leq w_k\leq 1$ and denote the resulting function by $h_2$. Now 

\begin{multline*}
	h_2 =216 + 17024 w_1 w_2 w_3 w_4 + 70 (w_1 + w_2 + w_3 + 
	w_4 + (w_1 - w_2)^2 + (w_1 - w_3)^2 + (w_1 - w_4)^2 +\\+ 
	(w_2 - w_3)^2 + (w_2 -	w_4)^2 + (w_3 - w_4)^2) + 
	112 ((w_3 + w_4) (w_2 - w_1)^2 + (w_2 + w_4) (w_1 - w_3)^2 + \\ + (w_2 + 
	w_3) (w_1 - w_4)^2 + (w_1 + w_4) (w_2 - w_3)^2 + (w_1 + 
	w_3) (w_2 - w_4)^2 + (w_2 + w_1) (w_4 - w_3)^2) + \\+
	56 ((w_1 w_2 + w_1 w_3 + w_2 w_3) ((5 w_4 - 1)^2 +  w_4^2) + (w_1 w_2 + w_1 w_4 +	w_2 w_4) ((5 w_3 - 1)^2 + w_3^2) + \\ + (w_1 w_3 + w_1 w_4 + 
	w_3 w_4) ((5 w_2 - 1)^2 +  w_2^2) + (w_2 w_3 + w_2 w_4 + 
	w_3 w_4) ((5 w_1 - 1)^2 + w_1^2)) + \\ + 14 ((4 w_1 w_2 - 1)^2 + (4 w_1 w_3 - 1)^2 + (4 w_1 w_4 - 1)^2 + (4 w_3 w_2 - 
	1)^2 + (4 w_4 w_2 - 1)^2 + (4 w_3 w_4 - 1)^2).
\end{multline*}

We can see that $h_2>0$, therefore $|b|>|a|$. The lemma is proved.

\end{proof}

Now we return to the proof of Proposition 1. It is easy to see that 
$$
f(1) = b-z_5a,
$$
where $a$ and $b$ are given by the same formulas as in Lemma~\ref{lem1}, but now we consider the case when $|z_k|\leq 1$, $1\leq k\leq 5$. 

In the paper \cite{hkk19} it was proved that $|b|>\frac16$ for all $z_k$ in the unit circle. 
Thus, $b\ne 0$ for all $z_k$. Then $\frac ab$ is an analytic function of four variables. From Lemma~\ref{lem1} we know that $ \left|\frac ab\right|<1$ when $|z_k|=1$, $1\leq k\leq 4$, so, applying the maximum principle to each variable, we see that the same inequality holds for $|z_k|\leq 1$. This is essentially the same argument that was given in \cite{hkk19} to prove a similar result there for polynomials of degree $6$. Therefore, $|b|>|a|$ and $|f(1)|>0$.

Proposition 1 is proved.

\section{Proof of the main result, Theorem~\ref{t1}}	
	Since $z_1, z_2, z_3, z_4, z_5$ lie on the unit circle, they can be represented as
	$$
	z_1=e^{i\phi_1},\,z_2=e^{i\phi_2},\,z_3=e^{i\phi_3},\,z_4=e^{i \phi_4},\,z_5=e^{i\phi_5}.
	$$
	We consider the expression 
	$$
	g(\phi_1,\phi_2,\phi_3,\phi_4,\phi_5)=25200\left(S^2-\frac1{49}\right).
	$$
	 If it is non-negative for all real values of the variables $\phi_k$ for $1\leq k\leq 5$, the desired inequality holds.

	For the sake of simplicity we write 
	$$
	\cos\phi_i-\cos\phi_j = x_{ij}, \quad \sin\phi_i-\sin\phi_j = y_{ij},\quad 1-\cos\phi_i = d_i,\quad 1-2\cos\phi_i = b_i.
	$$ 
		
We denote by $E_1$ the set of all permutations of the set $\{1,2,3,4,5\}$, that is, bijections $\varphi$ of $\{1,2,3,4,5\}$ onto itself.

We denote by $E_2$ the set of ordered pairs of unordered pairs of elements of $\{1,2,3,4,5\}$ such that all four elements of $\{1,2,3,4,5\}$ in each case are distinct. Then $E_2$ has $30$ elements, such as the set $[\{ 5,1 \}, \{ 2,4 \}]$, noting that in both pairs, as usual, the elements of a set are not in any particular order. Let us label the elements of $E_2$ as $e_1,\dots ,e_{30}$ in some definite way, and now choosing a definite order for the elements of  $\{1,2,3,4,5\}$ involved, for  each $e_j$, write $e_j$ in the form $e_j=[ \{ e_{j,1}, e_{j,2} \}, \{ e_{j,3} ,e_{j,4}  \} ]$. 

We denote by $E_3$ the set of all $10$ unordered pairs of distinct elements of $\{1,2,3,4,5\}$ and label them in some definite way as $T_j=\{ t_{j,1},t_{j,2}  \}$, choosing some definite order for the elements of each set $T_j$ , where $1\leq j\leq 10$.

We denote by $E_4$ the set of all ordered pairs of the form $(k_1,\{ k_2,k_3,k_4 \})$, where the numbers $k_1,k_2,k_3,k_4$ are distinct elements of $\{1,2,3,4,5\}$ and the triple $k_2,k_3,k_4$ is unordered, being expressed as a set. There are  $20$  elements of $E_4$ and we label them in some definite way as $U_j$ for $1\leq j\leq 20$. We write $U_j= (u_{j,1},\{ u_{j,2},u_{j,3},u_{j,4} \})$, where some (irrelevant but) definite order has been chosen in the triple. 

We define
\begin{eqnarray*} 
J_1 & = &  \frac{ 115 } { 6 }  \sum_{ \varphi \in E_1 } d_{ \varphi(1) } d_{ \varphi(2) } d_{ \varphi(3) } y_{  \varphi(4)   \varphi(5) }^2  + 
4 \sum_{j=1}^{30} x_{ e_{j,1}  e_{j,2} }^2   y_{ e_{j,3}  e_{j,4} }^2 
\\ & + & 
 \frac{ 11 } { 12 }  \sum_{ \varphi \in E_1 } ( b_{ \varphi(1) } + b_{ \varphi(2) } + b_{ \varphi(3) } )^2 d_{ \varphi(4) } d_{ \varphi(5) } 
\\ & + & 
 \frac{ 205 } { 48 }  \sum_{ \varphi \in E_1 } d_{ \varphi(1) } d_{ \varphi(2) } d_{ \varphi(3) }  ( b_{ \varphi(4) } + b_{ \varphi(5) }   )^2 
 \\ & + & 
 \frac{ 1 } { 144 }  \sum_{ \varphi \in E_1 } ( d_{ \varphi(1) } + d_{ \varphi(2) } + d_{ \varphi(3) } )  ( 19 b_{ \varphi(4) }^2  b_{ \varphi(5) }^2   +  96  y_{ \varphi(4) \varphi(5)  }^2     )     
   ,
\end{eqnarray*} 
\begin{eqnarray*} 
J_2 & = & \frac{1} { 16}  \sum_{ \varphi \in E_1 } \bigl(  2   (15 d_{ \varphi(1) } + 4 ) ( y_{  \varphi(2)  \varphi(3)  }^2 y_{  \varphi(4)   \varphi(5) }^2     ) 
\\
& + & 
\frac{3}{8}   d_{ \varphi(1) } (   b_{ \varphi(2) } + b_{ \varphi(3) } + b_{ \varphi(4) } + b_{ \varphi(5) }  )^2     + \frac{43}{12}   b_{ \varphi(1) }^2 ( d_{ \varphi(2) } + d_{ \varphi(3) } + d_{ \varphi(4) } + d_{ \varphi(5) }   )      \bigr) 
\\ & + &  
\frac{ 25 } { 4 } \sum_{j=1}^{20} b_{u_{j,1}}^2  d_{u_{j,2}} d_{u_{j,3}} d_{u_{j,4}}  ,   
\end{eqnarray*} 
$$
J_3 = \frac{1}{4} \sum_{j=1}^{30} (  b_{ e_{j,1} } b_{ e_{j,2} } - b_{ e_{j,3} } b_{ e_{j,4} }  )^2  ,  
$$
$$
J_4 = 2 \sum_{j=1}^{10} \left( 7 x_{ t_{j,1} t_{j,2} }^2 + 10 y_{ t_{j,1} t_{j,2} }^2    \right)   ,
$$
and
$$
J_5 = \frac{1} { 192 } \sum_{ \varphi \in E_1 } 
\left( 7  ( b_{ \varphi(1) }  + b_{ \varphi(2) } b_{ \varphi(3) } )^2 ( d_{ \varphi(4) }  + d_{ \varphi(5) }  )   + 720   d_{ \varphi(1) }^2 y_{ \varphi(2) \varphi(3) }^2    ( d_{ \varphi(4) }  + d_{ \varphi(5) }  )    \right) 
$$

Since $d_i \geq 0$ for all $i$, it is clear that $J_k\geq0$ for all $k$ with $k=1,\dots,5$. It can be  checked that
\begin{equation} \label{id1}
g(\phi_1,\phi_2,\phi_3,\phi_4,\phi_5)=J_1+J_2+J_3+J_4+J_5  .
\end{equation} 
Then $g\geq 0$, so $S^2\geq\frac1{49}$, which is equivalent to \eqref{source}. 

Moreover, the equality $g=0$ will hold only if each term $J_k$ is equal to zero. From the equality $J_4=0$ one can see that $x_{ij}=y_{ij}=0$ for all $1\leq i,j\leq 5$. This means that all the points $\phi_i$ are equal. Now it is evident that in this case  $d_i=0$ for all $i$ or $b_i=0$ for all $i$. This gives three points of minimum: 
$$
\phi_i=0,\quad i=1,\dots,5 ,  \qquad\phi_i=\frac{\pi}3,\quad i=1,\dots,5 , \qquad\text{and}\qquad\phi_i=-\frac{\pi}3,\quad i=1,\dots,5.
$$
In each of these cases, $g=0$, so $S=1/7$. This proves Theorem~\ref{t1}.

\begin{rem}
	Comparing this result with the one from the paper~\cite{hkk19}, one can see that the case $n=7$ is more complicated. In particular, for the counterpart of Problem~\ref{pr1} for $n=6$, which was implicitly proved in \cite{hkk19} as well, there exists only one class of extremal functions, while for $n=7$ we have three. 
	
	For $n=6$ all the extremal functions  were described by
	$$
	f(z) := \frac1{6a}(1-(1-az)^6),\qquad a\in\mathbb{C}\backslash \{0\}.
	$$
	For $n=7$ we get
	\begin{itemize}
		\item The function
		$$
		\tilde{g}_1(z) = z\int\limits_0^1(1-tz)^6\,dt
		$$
		and all the functions in the form $g_1(z) = \frac{\tilde{g}_1(az)}a$:
		$$
		g_1(z) := \frac1{7a}(1-(1-az)^7),\qquad a\in\mathbb{C}\backslash \{0\}.
		$$
		The equality $\left|\frac{f(z)}z\right| = \frac17$ will be attained  at the critical point $z=\frac1a$.  
		
		\item The functions
		$$
		\tilde{g}_{2,3}(z)=z\int\limits_0^1(1-tz)(1-tqz)^5\,dt,\qquad q=e^{\pm\frac{i\pi}3}
		$$
		and all the functions of the form $g_{2,3}(z) = \frac{\tilde{g}_{2,3}(az)}a$:
		$$
		g_{2,3}(z) :=\frac 1{42 q^2 a}((7q-1) (1 - (1 - a q  z)^6) +6 q a z (1 - a q z)^6),\qquad a\in\mathbb{C}\backslash \{0\}. 
		$$
		
		 As we have seen, the functions $g_{2,3}(z)$ are not extremal for the initial dual Smale's problem. However, they are extremal for Problem~\ref{pr1}.
		 
	\end{itemize}

\end{rem}

\section{Conclusion}
In this paper we continue the research on the dual Smale's mean value conjecture. The analytical proof of the  case $n=7$ is given. The common concept of the proof is the same as in the paper~\cite{hkk19}, but the methods are different.

The approach we use is based on  solving one auxiliary problem, denoted as Problem~\ref{pr1}. It turns out that for $n=7$ this problem has exactly three classes of extremal functions, which makes it different from all previous cases. 

We already know from \cite{hkk19} that for $n>7$ the same approach of choosing the arbitrary point with the minimal absolute value does not work. It means that for further investigations the method should be modified.

\section{Acknowledgments}
The work of D.M.\ Khammatova is supported by Russian Science Foundation, grant No.\ 22-71-10094.

\end{document}